\documentclass{amsart}
\usepackage{amsthm}
\usepackage{amsfonts}
\usepackage{amssymb}
\usepackage{amsmath}
\usepackage{nccmath}
\usepackage{mathtools}
\usepackage{mathrsfs}
\usepackage{setspace}
\usepackage{enumerate}
\newtheorem{theorem}{Theorem}[section]

\newtheorem{cor}{Corollary}[section]
\newtheorem{lemma}{Lemma}[section]
\theoremstyle{definition}
\newtheorem{defn}{Definition}
\theoremstyle{definition}

\theoremstyle{Remark}

\theoremstyle{proposition}
\newtheorem{prop}{Proposition}[section]

\begin{document}
\title[H(\MakeLowercase{b}) Spaces of Unit Ball of $\mathbb{C}^n$]
 {Angular Derivatives and Boundary Values of H(\MakeLowercase{b}) Spaces of Unit Ball of $\mathbb{C}^n$}

\author{S\.{I}bel \c{S}ah\.{I}n }

\address{Mimar Sinan Fine Arts University, Mathematics Department}

\email{sibel.sahin@msgsu.edu.tr}


\date{\today}



\begin{abstract}
In this work we study deBranges-Rovnyak spaces, $H(b)$, on the unit ball of $\mathbb{C}^n$. We give an integral representation of the functions in $H(b)$ through the Clark measure on $S^n$ associated with $b$. A characterization of admissible boundary limits is given in relation with finite angular derivatives. Lastly, we examine the interplay between Clark measures and angular derivatives showing that Clark measure associated with $b$ has an atom at a boundary point if and only if $b$ has finite angular derivative at the same point.
\end{abstract}

\maketitle
\section*{Introduction}

The theory of the Hardy spaces goes back to the beginning of 20th century and by the developments in functional analysis which treats these classes as linear spaces of holomorphic functions this special class became one of the central figures in the interplay between functional analysis and complex analysis. Among these $H^p$ classes the space $H^2(\mathbb{D})$ is of particular interest because $H^2(\mathbb{D})$ is a Hardy-Hilbert space and having a Hilbert space structure puts this space into the intersection of holomorphic function theory, functional analysis and operator theory as well. Later, from Beurling's exquisite solution to invariant subspace problem in $H^2(\mathbb{D})$, the subclasses called 'model spaces' emerged and they are of the form $\Theta H^2$ where $\Theta$ is an inner function in $H^2(\mathbb{D})$. A similar modeling theory was pioneered by L.deBranges and J.Rovnyak and it was the beginning of the theory of $H(b)$ spaces. In their approach the idea was to see $H(b)$ spaces as the complementary spaces just like in the model space case however as Sarason \cite{Sar} and many others pointed out $H(b)$ spaces can also be considered as the range of a specific contraction that contains Toeplitz operators. This approach became a determining point in the theory of deBranges-Rovnyak spaces because it enables the construction of $H(b)$ spaces on different regions other than unit disc and also to construct vector valued versions of these classes.

As it is well known in the 70's holomorphic function theory in several variables took a turn from sheaves etc and became more focused on boundary behaviour, $\overline{\partial}$-problem and smoothing. In his prominent work \cite{Rud}, Rudin considered the Hardy space theory in the setting of unit ball. Following his notation throughout this study, $H^{\infty}(\mathbb{B}^n)$ is the class of bounded holomorphic functions of the unit ball $\mathbb{B}^n\subset\mathbb{C}^n$ and the Hardy-Hilbert space $H^2(\mathbb{B}^n)$ is defined as
$$
H^2(\mathbb{B}^n)=\left\{f\in\mathcal{O}(\mathbb{B}^n):~~\sup_{0<r<1}\int_{S^n}|f(r\xi)|^2d\sigma(\xi)<\infty\right\}
$$
where the inner product on $H^2(\mathbb{B}^n)$ is given as:
$$
\langle f,g\rangle=\int_{S^n}f\overline{g}d\sigma,~~f,g\in H^2(\mathbb{B}^n).
$$

In this work we will consider deBranges-Rovnyak spaces of the unit ball of $\mathbb{C}^n$. The main focus will be on three topics, namely Clark measures, boundary limits of $H(b)$ functions and angular derivatives. First of all, as it was pointed out in \cite{Fri1}, $H(b)$ spaces are flexible in some sense compared to classical model spaces however when it comes to the representation it is a difficulty since the inner product does not have an explicit integral form however following their idea in the disc case we will give an integral representation of $H(b)$ functions of the unit ball through Clark measures. Since deBranges-Rovnyak space $H(b)$ is a subspace of the Hardy-Hilbert space $H^2(\mathbb{B}^n)$ we already know that $H(b)$ functions have radial boundary values (\cite{Rud}) and in this study we will give the full characterization of these boundary limits over admissible approach regions through finite angular derivative of the determining $H^\infty$ function $b$.

In the last part of this study we will consider the relation between Clark measure associated with $b$ and finite angular derivative of $b$ at the boundary $S^n$. The main result in this part will show us that the Clark measure associated with $b$ cannot have point mass at a point in $S^n$ as long as $b$ has infinite angular derivative at that point.

\section{Preliminaries}

Let $b\in H^\infty(\mathbb{B}^n)$, $b$ is a non-constant holomorphic function and $\|b\|_{\infty}\leq 1$. Analogous to the unit disc case let us define deBranges-Rovnyak space $H(b)$ of unit ball as the subspace of the Hilbert space $H^2(\mathbb{B}^n)$, defined by the inner product
\begin{equation}\label{innerproduct}
\|(I-T_{b}T_{\overline{b}})^{1/2}f\|_{b}=\|f\|_{2}~~(f\in H^2(\mathbb{B}^n)\ominus Ker(I-T_{b}T_{\overline{b}})^{1/2})
\end{equation}
Using this identification we say that $f\in H^2(\mathbb{B}^n)$ belongs to $H(b)$ if and only if
$$
\sup_{g\in H^2(\mathbb{B}^n)}(\|f+bg\|^{2}_{2}-\|g\|^{2}_{2})<\infty
$$
and
$$
\|f\|^{2}_{b}=\sup_{g\in H^2(\mathbb{B}^n)}(\|f+bg\|^{2}_{2}-\|g\|^{2}_{2}).
$$

As it can be seen from the definition, deBranges-Rovnyak spaces are more flexible than the classical model spaces however not having an integral representation for the inner product has its own difficulties hence it is important to be able to represent these spaces as reproducing kernel Hilbert spaces with an explicitly written kernel. Hence we will now focus on the reproducing kernel of the $H(b)$ spaces:

First of all, the reproducing kernel for the classical Hardy-Hibert space $H^2(\mathbb{B}^n)$ is
\begin{equation}\label{classicalkernel}
K(z,w)=\frac{1}{(1-\langle z,w\rangle)^n}
\end{equation}
and we will next show that the reproducing kernel for the $H(b)$ space can be written in terms of the classical kernel (\ref{classicalkernel}):
\begin{theorem}\label{reproducingkernel}
The reproducing kernel of $H(b)$ is
\begin{equation}\label{newrepkernel}
K^b(z,w)=(I-T_{b}T_{\overline{b}})K(z,w)=(1-\overline{b(z)}b(w))K(z,w)~~z,w\in\mathbb{B}^n
\end{equation}
or equivalently
\begin{equation}\label{newrepkernelexp}
K^b(z,w)=\frac{(1-\overline{b(z)}b(w))}{(1-\langle z,w\rangle)^n}
\end{equation}
\end{theorem}
\begin{proof}
By definition $H(b)=(I-T_{b}T_{\overline{b}})^{1/2}H^2(\mathbb{B}^n)$ and for all $f\in H^2(\mathbb{B}^n)$ we have
\begin{equation}\label{classicalrep}
f(z)=\int_{S^n}\frac{f(\xi)}{(1-\langle z,\xi\rangle)^n}d\sigma(\xi)
\end{equation}
so by (\cite{Fri2}, Theorem 16.13) we have
$$
K^b(z,w)=(I-T_{b}T_{\overline{b}})K(z,w).
$$
As it can be easily verified (for more details see \cite{Zhu}), we have
$$
T_{\overline{b}}K(z,w)=\overline{b(w)}K(z,w)
$$
and
$$
T_{b}K(z,w)=b(z)K(z,w)
$$
hence
$$
K^b(z,w)=\frac{(1-\overline{b(z)}b(w))}{(1-\langle z,w\rangle)^n}.
$$
\end{proof}

Now since
$$
f(z)=\langle f,K(z,.)\rangle_{H^2(\mathbb{B}^n)}=\langle f,K^b(z,.)\rangle_{b}
$$
and $\|K^{b}(z,z)\|^{2}_{2}=K^b(z,z)$, for all $z\in \mathbb{B}^n$ we have
$$
\|K^b(z,z)\|_{b}=\left(\frac{1-|b(z)|^2}{(1-\|z\|^2)^n}\right)^{^1/2}.
$$

\section{Integral Representation of $H(b)$ and Clark Measures}

Although we have the explicit formulation of the reproducing kernel $K^b(z,w)$ we still cannot represent an $H(b)$ function as in (\ref{classicalrep}) since the inner product $\langle . , .\rangle_b$ is given implicitly with respect to the classical inner product. However, one can still have an integral representation for $H(b)$ spaces using Clark measures:

Let $z\in\mathbb{B}^n$, $\xi\in S^n$ with $\langle z,\xi\rangle\neq 1$ then the equality
$$
C(z,\xi)=\frac{1}{(1-\langle z,\xi\rangle)^n}
$$
defines the Cauchy kernel for $\mathbb{B}^n$. Then, the invariant Poisson kernel is given by the formula
$$
P(z,\xi)=\frac{C(z,\xi)C(\xi,z)}{C(z,z)}=\left(\frac{1-\|z\|^2}{|1-\langle z,\xi\rangle|^2}\right)^n.
$$
\begin{defn}
Given an $\alpha\in \mathbb{T}$ and a holomorphic function $\varphi:\mathbb{B}^n\rightarrow\mathbb{D}$ the quotient
$$
Re\left(\frac{\alpha+\varphi(z)}{\alpha-\varphi(z)}\right)=\frac{1-|\varphi(z)|^2}{|\alpha-\varphi(z)|^2}
$$
is positive and pluriharmonic therefore there exists a unique positive measure $\mu$ such that
$$
\int_{S^n}P(z,\xi)d\mu(\xi)=Re\left(\frac{\alpha+\varphi(z)}{\alpha-\varphi(z)}\right).
$$
\end{defn}

The measure $\mu$ is called the Clark measure associated to $\varphi$.

\begin{lemma}
Let $\mu$ denote the Clark measure of $b$. Then for $z,w\in \mathbb{B}^n$,
$$
\langle K(.,w),K(.,z)\rangle_\mu=\int_{S^n}K(\xi,w)\overline{K(\xi,z)}d\mu(\xi)=\frac{K^b(z,w)}{(1-\overline{b(w)})(1-b(z))}.
$$
\end{lemma}
\begin{proof}
Take $\alpha=1\in \mathbb{T}$ then by (Proposition 2.2,\cite{AlexDubs}) we have
$$
\int_{S^n}C(z,\xi)C(\xi,w)d\mu(\xi)=\frac{1-b(z)\overline{b(w)}}{(1-b(z)(1-\overline{b(w)}))}C(z,w)
$$
[Note that $I$ being inner plays no role in the proof of (Proposition 2.2,\cite{AlexDubs}) so we can apply it to the case of $b$.] Hence from the definition of $K^b(z,w)$ and the Clark measure we obtain the result.
\end{proof}

Since we have the relation
$$
K^b(z,w)=\langle K^b(z,w),K^b(z,z)\rangle_b
$$
the equation
$$
\langle K(.,w),K(.,z)\rangle_\mu=\frac{K^b(z,w)}{(1-\overline{b(w)})(1-b(z))}
$$
can be written as
\begin{equation}\label{clarkinnerproduct}
\langle K(.,w),K(.,z)\rangle_\mu=\langle \frac{K^b(.,w)}{1-\overline{b(w)}}, \frac{K^b(.,z)}{1-\overline{b(z)}}\rangle_b.
\end{equation}
For the Clark measure $\mu$, let us define the following integral operator
\begin{equation}\label{clarkintegralrep}
K_\mu f(z)=\int_{S^n}\frac{f(\xi)}{(1-\langle z,\xi\rangle)^n}d\mu(\xi)
\end{equation}
then by (Lemma 13.9,\cite{Fri1}) $V_b(f)(z)=(1-b(z))K_\mu f(z)$ is a well-defined continuous operator from $L^2(\mu)$ into $\mathcal{O}(\mathbb{B}^n)$. [Since the lemma applies directly to $\mathbb{B}^n$ case we do not repeat the proof here.]

The relation (\ref{clarkinnerproduct}) suggests that
$$
H^2(\mu)\rightarrow H(b)
$$
$$
K(z,w)\rightarrow\frac{K^b(z,w)}{1-\overline{b(w)}}
$$
is related to $V_b(f)(z)$ :
\begin{theorem}\label{maintheorem1}
The mapping $V_b$ is a partial isometry from $L^2(\mu)$ onto $H(b)$ and $KerV_b=(H^2(\mu))^{\bot}$. Moreover,
$$
V_bK(z,w)=\frac{K^b(z,w)}{1-\overline{b(w)}},~~w\in\mathbb{B}^n.
$$
\end{theorem}
\begin{proof}
First of all by the previous lemma and (\ref{clarkintegralrep}) we have,
\begin{eqnarray}
V_bK(z,w)&=&(1-b(z))K_\mu K(.,w)(z)\nonumber\\
&=& (1-b(z))\int_{S^n}\frac{K(\xi,w)}{(1-\langle z,\xi\rangle)^n}d\mu(\xi)\nonumber\\
&=& (1-b(z))\int_{S^n}\frac{1}{((1-\langle \xi,w\rangle)^n)(1-\langle z,\xi\rangle)^n}d\mu(\xi)\nonumber\\
&=& (1-b(z))\langle K(.,w),K(.,z)\rangle_\mu =\frac{K^b(z,w)}{1-\overline{b(w)}}\label{operatorkernelrep}
\end{eqnarray}
and by (\ref{clarkinnerproduct}) we get,
$$
\langle K(.,w),K(.,z)\rangle_\mu=\langle V_bK(.,w),V_bK(.,z)\rangle_b.
$$
In particular $\|g\|_{L^2(\mu)}=\|V_bg\|_b$ when $g$ is a finite linear combination of kernel functions $K(.,w_1),\dots,K(.,w_n)$ for some $w_1,\dots,w_n\in \mathbb{B}^n$. For a $g\in H^2(\mu)$, the generalized Hardy space [for details see \cite{Fri1}, chapter 5] by (Theorem 5.11, \cite{Fri1}) we know that there exists a sequence $\{g_n\}$ converging to $g\in H^2(\mu)$ where each $g_n$ is a finite linear combination of kernel functions. Since $V_b$ is continuous we have $V_bg_n\rightarrow V_bg$ both in the topology of $\mathcal{O}(\mathbb{B}^n)$ and pointwise in $\mathbb{B}^n$. Since $V_bg_n$ is a Cauchy sequence there exists $f\in H(b)$ such that $V_bg_n\rightarrow f$ and by continuity of evaluations on $H(b)$ for $z\in \mathbb{B}^n$ $V_bg_n(z)\rightarrow f(z)$ and $V_bg=f$. Now since $\|g\|_{L^2(\mu)}=\|V_bg\|_b$ for all finite linear combinations one gets
$$
\|V_bg\|=\|f\|_b=\lim_{n\rightarrow \infty}\|V_bg_n\|_b=\lim_{n\rightarrow \infty}\|g_n\|_{L^2(\mu)}=\|g\|_{L^2(\mu)}.
$$
Hence $V_b$ is an isometry and by (\ref{operatorkernelrep}) range of $V_b$ contains all elements $K^b(z,w)$, $w\in\mathbb{B}^n$ so $V_bH^2(\mu)=H(b)$. Then we have
$$
V_b:L^2(\mu)\rightarrow H(b)
$$
$$
g\rightarrow(1-b)K_\mu g
$$
is a partial isometry so the operator
$$
V_b:H^2(\mu)\rightarrow H(b)
$$
$$
g\rightarrow(1-b)K_\mu g
$$
is a unitary operator and for an $f\in H(b)$ there exists a unique $g\in H^2(\mu)$ such that $f=V_bg$ and
\begin{equation}\label{integralrepofH(b)}
f(z)=(1-b(z))\int_{S^n}\frac{g(\xi)}{(1-\langle z,\xi\rangle)^n}d\mu(\xi),~~z\in\mathbb{B}^n.
\end{equation}

\end{proof}

\section{Restricted Limits and Finite Angular Derivatives of $H(b)$ Functions of $\mathbb{B}^n$}

$H(b)$ functions are a subclass of the Hardy-Hilbert space $H^2(\mathbb{B}^n)$ and Hardy space functions have boundary values through specific approach regions. In this section we will consider the restricted limits of $H(b)$ class and their relation to the finite angular derivatives.

\begin{defn}
We say $f$ has restricted limit at $\xi$ on $S^n$ if $f$ has limit $f^*(\xi)$ along every curve $\Lambda(t)$ approaching $\xi$ that satisfies
$$
(1)\lim_{t\rightarrow1}\frac{|\Lambda(t)-\lambda(t)|^2}{1-|\lambda(t)|^2}=0
$$
and
$$
(2)\frac{|\lambda(t)-\xi|}{1-|\lambda(t)|}\leq M<\infty~~\text{for}~~0\leq t<1
$$
where $\lambda(t)$ is the orthogonal projection of $\Lambda(t)$ onto the complex line $[\xi]$ through $0$ and $\xi$ and $\lambda(t)=\langle\Lambda(t),\xi\rangle$.
\end{defn}

For $\alpha>1$, let
$$
\Gamma(\xi,\alpha)=\{z\in\mathbb{B}^n:~~|1-\langle z,\xi\rangle|<\frac{\alpha}{2}(1-\|z\|^2)\},
$$
\begin{theorem}[\cite{Cow}, 2.79]
Suppose $f$ is holomorphic in $\mathbb{B}^n$ and bounded in every approach region $\Gamma(\xi,\alpha)$. If $\lim_{r\rightarrow 1}f(r\xi)$ exists then $f$ has restricted limit at $\xi$.
\end{theorem}
Now let us consider the angular derivatives for the class $H(b)$:
\begin{defn}
We say $\varphi: \mathbb{B}^n\rightarrow\mathbb{B}^m$ has finite angular derivative at $\xi\in S^n$ if there exists $\eta\in S^m$ so that
$$
\frac{\langle \varphi(z)-\eta,\eta\rangle}{\langle \xi-z,\xi\rangle}
$$
has finite restricted limit at $\xi$.
\end{defn}

\begin{defn}
We say that a function has angular derivative in the sense of Carath\'{e}odory at $\xi\in S^n$ if it has finite angular derivative at $\xi$ and moreover $|f(\xi)|=1$.
\end{defn}

The following theorem is the main result of this section that characterizes the admissible limits of $H(b)$ functions through angular derivatives:
\begin{theorem}\label{maintheorem2}
Let $b$ be holomorphic in $\mathbb{B}^n$, let $\xi\in S^n$ and put
$$
c=\lim_{z\rightarrow\xi} \frac{1-|b(z)|}{1-\|z\|}.
$$
Then the following are equivalent
\begin{itemize}
\item[(i)] $c$ is finite
\item[(ii)] There is $\eta\in\mathbb{T}$ such that
$$
\frac{\eta-b(z)}{(1-\langle z,\xi\rangle)^n}\in H(b)
$$
\item[(iii)] For all functions $f\in H(b)$, $f$ has admissible limit at $\xi$ (i.e it has a limit $f^*(\xi)$ along every curve lying in some admissible approach region $\Gamma(\xi,\alpha)$).
\item[(iv)] $b$ has angular derivative in the sense of Carath\'{e}odory.
\end{itemize}
\end{theorem}
\begin{proof}

($i\Rightarrow ii$) If $c<\infty$ then there is a sequence $(z_n)\in \mathbb{B}^n$ converging to $\xi$ such that
$$
c=\lim_{n\rightarrow\infty} \frac{1-|b(z_n)|}{1-\|z_n\|}<\infty
$$
so we have $\lim_{n\rightarrow\infty}|b(z_n)|=1$ and we can write
$$
c=\lim_{n\rightarrow\infty}\frac{1-|b(z_n)|^2}{1-\|z_n\|^2}
$$
which actually is
$$
c=\lim_{n\rightarrow\infty}\|K^b(z,z_n)\|^{2}_{b}.
$$
Then $(K^b(z,z_n))$ has a weakly convergent subsequence in $H(b)$. Since $(b(z_n))_{n\geq 1}$ is bounded, it also has a convergent subsequence in $\overline{\mathbb{B}^n}$ hence we may assume there are $\eta\in \overline{\mathbb{D}}$ and $k\in H(b)$ such that $b(z_n)\rightarrow\eta$ and $K^b(z,z_n)\xrightarrow{w}k$.
For each $z\in \mathbb{B}^n$,
$$
k(z)=\langle k,K^b(z,z)\rangle_b=\lim_{n\rightarrow\infty}\langle K^b(z,z_n),K^b(z,z)\rangle_b=\lim_{n\rightarrow\infty}K^b(z,z_n)
$$
$$
=\lim_{n\rightarrow\infty}\frac{1-\overline{b(z_n)}b(z)}{(1-\langle z_n,z\rangle)^n}=\frac{1-\overline{\eta}b(z)}{(1-\langle \xi,z\rangle)^n}.
$$

Since $k\in H^2(\mathbb{B}^n)$ and $1/(1-\langle \xi,z\rangle)^n$ we have $|\eta|=1$ and
$$
\eta k(z)=\frac{\eta-b(z)}{(1-\langle \xi,z\rangle)^n}\in H(b).
$$
Now since $k\not\equiv0$, $K^b(z,z_n)\xrightarrow{w}k$ implies
$$
0<\|k\|^{2}_{b}\leq\liminf_{n\rightarrow\infty}\|K^b(z,z_n)\|_{b}^{2}=c.
$$

($ii\Rightarrow iii$)
Assume that $k\in H(b)$ then since $k\in H^2(\mathbb{B}^n)$
$$
b(z)=\eta-\eta k(z)(1-\langle \xi,z\rangle)^n
$$
implies that
\begin{eqnarray}
|b(z)-\eta|&\leq&\|k\|_2\|K^b(z,z)\|_2|1-\langle \xi,z\rangle|^n\\
&=&\|k\|_2\left(\frac{1-|b(z)|^2}{(1-\|z\|^2)^n}\right)^{1/2}|1-\langle \xi,z\rangle|^n
\end{eqnarray}
thus if $z\in\Gamma(\xi,\alpha)$ for some $\alpha>1$ then,
\begin{eqnarray}
|b(z)-\eta|&\leq&\|k\|_2\left(\frac{1-|b(z)|^2}{(1-\|z\|^2)^n}\right)^{1/2}\frac{\alpha}{2}(1-\|z\|^2)^n\\
&=&\|k\|_2\frac{\alpha}{2}(1-|b(z)|^2)^{1/2}(1-\|z\|^2)^{n/2}
\end{eqnarray}
and as $z\rightarrow\xi$ from $\Gamma(\xi,\alpha)$ the right hand side tends to $0$ so $\lim_{\substack{z\rightarrow\xi\\\Gamma(\xi,\alpha)}}b(z)=\eta$.

Now denoting $\eta$ as $\beta(\xi)$ write
$$
K^b(z,\xi)=\frac{1-\overline{b(\xi)}b(z)}{(1-\langle \xi,z\rangle)^n}
$$
so
$$
K^b(z,\xi)=\langle K^b(z,\xi),K^b(z,z)\rangle_b
$$
and
\begin{equation}\label{equationstar}
|K^b(z,\xi)|\leq\|K^b(.,\xi)\|_b\|K^b(.,z)\|_b
\end{equation}
Also,
\begin{equation}\label{equationdoublestar}
|K^b(z,\xi)|=\frac{|1-\overline{b(\xi)}b(z)|}{|1-\langle z,\xi\rangle|^n}\geq\frac{1-|b(z)|}{|1-\langle z,\xi\rangle|^n}=\frac{(1-\|z\|^2)^n\|K^b(.,z)\|_{b}^{2}}{(1+|b(z)|)|1-\langle z,\xi\rangle|^n}
\end{equation}
Combining (\ref{equationstar}) and (\ref{equationdoublestar}) we get
$$
\|K^b(.,z)\|_b\leq\|K^b(.,\xi)\|_b\frac{(1-|b(z)|)}{(1-\|z\|^2)^n}|1-\langle z,\xi\rangle|^n
$$
so in an admissible approach region $\Gamma(\xi,\alpha)$
$$
\|K^b(.,z)\|_b\leq A\|K^b(.,\xi)\|_b~~(z\in \Gamma(\xi,\alpha)).
$$
Now for $u\in\mathbb{B}^n$,
$$
\lim_{\substack{z\rightarrow\xi\\\Gamma(\xi,\alpha)}}K^b(u,z)=\lim_{\substack{z\rightarrow\xi\\\Gamma(\xi,\alpha)}}\frac{1-\overline{b(z)}b(u)}{(1-\langle u,z\rangle)^n}=\frac{1-\overline{b(\xi)}b(u)}{(1-\langle u,\xi\rangle)^n}=K^b(u,\xi)
$$
which is
$$
\lim_{\substack{z\rightarrow\xi\\\Gamma(\xi,\alpha)}}\langle K^b(.,z),K^b(.,u)\rangle_b=\langle K^b(.,\xi),K^b(.,u)\rangle_b.
$$
Hence,
\begin{equation}\label{equationtriplestar}
\lim_{\substack{z\rightarrow\xi\\\Gamma(\xi,\alpha)}}\langle f,K^b(.,z)\rangle_b=\langle f,K^b(.,\xi)\rangle_b
\end{equation}
where $f\in H(b)$ is any element of the form $f=\alpha_1K^b(.,w_1)+\dots+\alpha_nK^b(.,w_n)$ and since the elements of this sort are dense in $H(b)$ (\ref{equationtriplestar}) holds for all $f\in H(b)$ and
$$
f(\xi)=\lim_{z\rightarrow\xi}f(z)=\langle f,K^b(.,\xi)\rangle_b,~~f\in H(b).
$$

($iii\Rightarrow i$) Take an admissible approach region $\Gamma(\xi,\alpha)$ then we have
$$
\sup_{z\in \Gamma(\xi,\alpha)}|\langle f,K^b(.,z)\rangle|=C(f)<\infty
$$
so by Uniform Boundedness principle
$$
\widetilde{C}=\sup_{z\in \Gamma(\xi,\alpha)}\|K^b(.,z)\|_b<\infty.
$$
Consider $z_n=(1-1/n\|)\xi$, $n\geq 1$, then $z_n\in \Gamma(\xi,\alpha)$ for big $n$ and we get
$$
\frac{1-|b(z_n)|^2}{1-\|z_n\|^2}=\|K^b(.,z_n)\|_{b}^{2}\leq \widetilde{C}^2
$$
which also gives $\lim_{n\rightarrow\infty}|b(z_n)|=1$ and
$$
C\leq \liminf_{n\rightarrow\infty}\frac{1-|b(z_n)|^2}{1-\|z_n\|^2}=\liminf_{n\rightarrow\infty}\|K^b(.,z_n)\|_{b}^{2}\leq\widetilde{C}^2.
$$

($i\Rightarrow iv$) This part directly follows from the following theorem by \cite{Rud},
\begin{theorem}
Let $f:\mathbb{B}^n\rightarrow\mathbb{B}^m$ be a holomorphic map such that
$$
\liminf_{z\rightarrow p}\displaystyle{\frac{1-\|f(z)\|}{1-\|z\|}=\alpha}<\infty
$$
for some $p\in S^n$. Then $f$ admits admissible limit $q\in S^m$ at $p$ and furthermore for all $z\in \mathbb{B}^n$,
$$
\frac{|1-\langle f(z),q\rangle|^2}{|1-\langle z,p\rangle|^2}\leq \alpha\frac{1-\|f(z)\|^2}{1-\|z\|^2}.
$$
\end{theorem}

($iv\Rightarrow i$) If $b$ has finite angular derivative in the sense of Carath\'{e}odory at $\xi$ then
$$
\frac{1-|b(r\xi)|}{1-r}\leq\frac{|b(r\xi)-b(\xi)|}{|\langle r\xi-\xi,\xi\rangle|}
$$
and
$$
c=\liminf_{z\rightarrow\xi}\frac{1-|b(z)|}{1-\|z\|}\leq\lim_{r\rightarrow1}\frac{|b(r\xi)-b(\xi)|}{|\langle r\xi-\xi,\xi\rangle|}=|b'(\xi)|<\infty.
$$

\end{proof}

\begin{cor}
If the composition operator with symbol $\varphi$, $C_\varphi$, is compact on $H(b)$ and $\displaystyle{\frac{1-|b(\varphi(w_\ell))|^2}{1-|b(w_\ell)|^2}}$ is uniformly bounded for any sequence $w_\ell\rightarrow S^n$ then $\varphi$ cannot have finite angular derivative at any point of $S^n$.
\end{cor}
\begin{proof}
Let $K^b(.,w)$ be the kernel function for $H(b)$. Since
$K^b(z,w)=\displaystyle{\frac{1-\overline{b(w)}b(z)}{(1-\langle z,w\rangle)^n}}$ and $\|K^b(.,w)\|^2=\displaystyle{\frac{1-|b(w)|^2}{(1-\|w\|^2)^n}}$ if we take a sequence $(w_\ell)\rightarrow S^n$ then $k_j=\frac{K^b(.,w_j)}{\|K^b(.,w_j)\|}\rightarrow 0$ as $j\rightarrow\infty$. If $T$ is a compact operator on $H(b)$ then $T(k_j)\rightarrow 0$ and $\|C_\varphi\|_e=\inf\{\|C_\varphi-P\|,~~P~~\text{is compact}\}$ so
$$
\|C_\varphi-T\|\geq \limsup_{j\rightarrow\infty}\|(C_\varphi-T)(k_j)\|=\limsup_{j\rightarrow\infty}\|C_\varphi(k_j)\|=\limsup_{j\rightarrow\infty}\frac{\|K^b(.,\varphi(w_j))\|}{\|K^b(.,w_j)\|}.
$$
Hence we have,
$$
\|C_\varphi\|_e\geq\limsup_{\|w\|\rightarrow1}\frac{\|K^b(.,\varphi(w))\|}{\|K^b(.,w)\|}
$$
which gives that if $C_\varphi$ is compact then $\left(\frac{1-\|w\|^2}{1-\|\varphi(w)\|^2}\right)^n\rightarrow0$ for any sequence approaching to $S^n$ hence $\varphi$ cannot have finite angular derivative at any point in $S^n$ by the previous theorem.
\end{proof}

We will finish this section with the relation between Clark measures and finite angular derivatives:

\begin{prop}
Let $b$ be a non-constant function in the closed unit ball of $H^\infty(\mathbb{B}^n)$ and let $\Theta$ be a non-constant inner function. Let $\mu$ and $\nu$ denote respectively the Clark measures of $b$ and $\Theta$. Then the following are equivalent:
\begin{itemize}
\item[(i)] $\nu$ is absolutely continuous with respect to $\mu$ and $\frac{d\nu}{d\mu}\in L^2(\mu)$
\item[(ii)] $\frac{1-b}{1-\Theta}K^\Theta(z,0)\in H(b)$.
\end{itemize}
\end{prop}
\begin{proof}
The result follows the same lines of the unit disc case (Theorem 20.28, \cite{Fri2}) combined with (\ref{maintheorem1}) so we will not repeat the same argument here.
\end{proof}
\begin{prop}
Let $b$ be a non-constant function in the closed unit ball of $H^\infty(\mathbb{B}^n)$ and $\mu_\eta$ be the corresponding Clark measure for some $\eta\in \mathbb{T}$. Let $\xi_0\in S^n$ then
$$
\frac{\eta-b(z)}{(1-\langle z,\xi_0\rangle)^n}\in H(b)
$$
if and only if $\mu_\eta(\{\xi_0\})>0$.
\end{prop}
\begin{proof}
First of all by purely measure theoretic reasons one can see that $\mu_\eta(\{\xi_0\})>0$ if and only if $\delta_{\xi_{0}}$ is absolutely continuous with respect to $\mu_\eta$ and $\frac{d\delta_{\xi_{0}}}{d\mu_\eta}\in L^2(\mu_\eta)$ where $\delta_{\xi_{0}}$ is the Dirac measure associated with $\xi_{0}$. Since $\delta_{\xi_{0}}$ on $S^n$ is the Clark measure of $\Theta(z)=1-(1-\langle z,\xi_0\rangle)^n$ by the previous proposition we have that
$$
\frac{1-b(z)}{1-\Theta(z)}K^\Theta(z,0)\in H(b)
$$
i.e
$$
\frac{1-b(z)}{(1-\langle z,\xi_0\rangle)^n}(1-\overline{\Theta(0)}\Theta(z))\in H(b)
$$
which in turn gives that $\displaystyle{\frac{1-\overline{\eta}b(z)}{(1-\langle z,\xi_0\rangle)^n}}\in H(\overline{\eta} b)$ and this is equivalent to the condition that
$$
\frac{\eta-b(z)}{(1-\langle z,\xi_0\rangle)^n}\in H(b).
$$

\end{proof}
\begin{cor}
The function $b$ has finite angular derivative in the sense of Carath\'{e}odory at the point $\xi_0\in S^n$ if and only if there exists $\eta\in \mathbb{T}$ such that $\mu_\eta(\{\xi_0\})>0$.
\end{cor}

\end{document}